\documentclass{article}
\usepackage{graphics}
\usepackage{graphicx}
\usepackage{latexsym}
\usepackage{amssymb}
\usepackage{amsmath}

\newtheorem{theorem}{Theorem}[section]
\newtheorem{lema}[theorem]{Lemma}
\newtheorem{prop}[theorem]{Proposition}
\newtheorem{corollary}[theorem]{Corollary}
\newtheorem{definition}{Definition}[section]
\newtheorem{preexample}{Example}[section]

\newtheorem{preremark}{Remark}

\newcommand{\qed }{ \hfill $\Box$ }

\begin{document}

\begin{center}
{\Large Vertical martingales, stochastic calculus and harmonic sections\\}

\end{center}

\vspace{0.3cm}

\begin{center}
{\large  Sim\~ao Stelmastchuk}  \\

\textit{Departamento de  Matem\'atica, Universidade Estadual do Paran\'a,\\ 84600-000 -  Uni\~ao da Vit\' oria - PR,
Brazil. e-mail: simnaos@gmail.com}
\end{center}

\vspace{0.3cm}

\begin{abstract}
This work is about a new class of martingales: the vertical martingales. We construct the vertical martingale for
smooth submersions and we develop a stochastic calculus for one. Furthermore, we give a stochastic characterization
for harmonic sections.
\end{abstract}


\vspace{0.3cm} \noindent {\bf MSC2010 subject classification:} 53C43, 58E20, 58J65, 60G48

\section{Introduction}

Let $\pi:E\rightarrow M$ be a Riemannian submersion with the totally geodesic fibers property and denote by $\nabla^{g}$
the Levi-Civita connection on $E$.  Since $\pi$ is a submersion, it is possible to define the vertical spaces by
$V_{p}E = \mathrm{ker}(\pi_{*p})$, $p \in E$. We also define the vertical connection $\nabla^{v}$ on $E$ by vertical
projection of the Levi-Civita connection $\nabla^{g}$ into $V_{p}E$. With this hypothesis, in
\cite{wood1}, C.Wood defines a harmonic section being a section $\sigma$ of $\pi$ such that
\[
\tau_{\sigma}^{v}= \mathrm{tr}\nabla^{v}\sigma_{*}^v = 0,
\]
where $\sigma_{*}^{v}$ is the vertical projection of $\sigma_{*}$ into $VE$. In fact, Wood shows that this definition
is consistent with a minimal solution of the vertical energy functional.

Our work has its main idea based on harmonic sections. We explain it. It is well-known that there is a stochastic characterization
for harmonic maps (see for example \cite{catuogno2} or \cite{emery1}). In a nutshell, if $M$ is a Riemannian manifold,
$N$ a smooth manifold with a symmetric connection and $\phi: M \rightarrow N$ a smooth map, then $\phi$ is a harmonic
map if and only if $\phi$ sends Brownian motions in $M$ into martingales in $N$. However, harmonic sections ask the
vertical connection on the target manifold. So, to construct a stochastic characterization for harmonic sections is
necessary a new concept of the martingale: the vertical martingale.

The environment of our work is more general that is used by Wood. Let $E, M$ be smooth manifolds such that there exists a submersion $\pi:E \rightarrow M$. We also endow $E$ with a symmetric connection. The stochastic calculus on manifolds says that to define a martingale is necessary the concept of the integral of It\^o. In this way, before to define the vertical martingales we need to construct the vertical It\^o integral. Based in the Rank theorem, we can define the vertical It\^o integral using the Schwartz Theory. Furthermore, we define the vertical Stratonovich integral and to show a formula of changes between both.

A way to show the stochastic characterization of harmonic maps is to use the geometric It\^o formula (see for example \cite{catuogno2}). In the same line, we construct a geometric It\^o formula for the vertical It\^o integral and the vertical Stratonovich integral. Both are useful. The geometric It\^o formula is used to give a stochastic characterization for the vertical harmonic maps.

The stochastic characterization of the vertical harmonic map directly gives that a section $\sigma:M \rightarrow E$ of $\pi$, where $M$ is a Riemannian manifold, is harmonic section if and only if $\sigma$ sends Brownian motions  into vertical martingales.

As applications we study the vertical martingales in the tangent space $TM$ endowed with the complete lift connection or the Sasaky metric and, consequently, we conclude that every harmonic section with values in $TM$ is the 0-section. Furthermore, we study the vertical martingales in the Riemannian principal fiber bundle.

\section{Preliminaries}

We begin by recalling some fundamental facts on Schwartz Theory and stochastic calculus on manifolds. We shall freely use concepts and notations from S. Kobayashi and N. Nomizu \cite{kobay}, L. Schwartz \cite{schwartz}, P.A. Meyer \cite{meyer} and M. Emery \cite{emery1}. A quick survey in these concepts is described by P. Catuogno in \cite{catuogno2}.

Let $M$ be a smooth manifold and $x \in M$. The second order tangent space to $M$ at $x$, which is denoted by $\tau_{x}M$, is the vector space of all differential operators on M at x of order at most two without a constant term. Taking a local system of coordinates $(x_{1}, \ldots, x_{n})$ around at x we can write every $L \in \tau_{x}M$, in a unique way, as
\[
 L = a_{ij}D_{ij} + a_{i}D_{i},
\]
where $a_{ij}=a_{ji}$, $D_{i}= \frac{\partial}{\partial x^{i}}$ and $D_{ij}=\frac{\partial^{2}}{\partial x^{i}\partial
x^{j}}$ are differential operators at x (we shall use the convention of summing over repeated indices). The elements of
$\tau_{x}M$ are called second order tangent vectors at $x$, the elements of the dual vector space $\tau_{x}^{*}M$ are
called second order forms at $x$.

The disjoint union $\tau M = \bigcup_{x \in M} \tau_{x}M$ (respectively, $\tau^{*}M  = \bigcup_{x \in M}
\tau_{x}^{*}M$) is canonically endowed with a vector bundle structure over $M$, which is called the second order
tangent fiber bundle (respectively, second order cotangent fiber bundle) of $M$.

Let $M, N$ be smooth manifolds, $F :M\rightarrow N$ a smooth map and $L\in \tau _{x}M$. The differential of $F$ ,
$F_{*}(x):\tau_{x}M \rightarrow \tau_{F(x)}N$, is given by
\[
F _{*}(x)L(f)=L_{x}(f\circ F ),
\]
where $f \in \mathcal{C}^{\infty}(N)$.

Let $L$ be a second order vector field on $M$. The square operator of $L$, denoted by $QL$, is the symmetric tensor
given by
\[
QL(f,g)=\frac{1}{2}(L(fg)-fL(g)-gL(f)),
\]
where $f,g\in C^{\infty }(M)$. Let $x \in M$. We consider $Q_{x}:\tau_{x}M\rightarrow T_{x}M\odot T_{x}M$ as the linear
application defined by
\[
Q_{x}(L = a_{ij}D_{ij} + a_{i}D_{i}) = a_{ij}D_{i} \odot D_{j}.
\]

Push-forward of the second order vectors by smooth maps is related to the so called Schwartz morphisms between second order tangent vector bundles.
%
\begin{definition}
Let $M$ and $N$ be smooth manifolds, $x\in M$ and $y\in N$. A linear application $F:\tau_{x}M\rightarrow \tau _{y}N$ is called Schwartz morphism if
\begin{enumerate}
\item  $F(T_{x}M)\subset T_{y}N$;
\item  for all $L\in \tau _{x}M$ we have $Q(FL)=(F\otimes F)(QL)$.
\end{enumerate}
\end{definition}

A linear application $F:\tau_{x}M \rightarrow \tau_{y}N$ is a Schwartz morphism if and only if there exists a smooth
map $\phi: M \rightarrow N$ with $\phi(x)=y$ such that $F=\phi_{x*}$ (see for example Proposition 1 in \cite{emery2}).

Let $(\Omega, (\mathcal{F}_t),\mathbb{P})$ be a filtered probability space which satisfies the usual conditions (see
for instance \cite{emery1}).
\begin{definition}
Let $M$ be a smooth manifold and $X$ a stochastic process with values in $M$. We call $X$ a semimartingale if, for
all $f$ smooth on $M$, $f(X)$ is a real semimartingale.
\end{definition}

L. Schwartz has noticed, in \cite{schwartz}, that if $X$ is a continuous semimartingale in a smooth manifold $M$, the
It\^o's differentials $dX_{i}$ and $d[X_{i},X_{j}]$ (where $(x_{i})$ is a local coordinate system and $X_{i}$ is the
ith coordinate of $X$ in this system) behave under a change of coordinates as the coefficients of a second order
tangent vector. The (purely formal) stochastic differential
\[
 d^{2}X_{t} = dX^{i}_{t}D_{i} + \frac{1}{2}d[X^{i},X^{j}]_{t}D_{ij},
\]
is a linear differential operator on $M$, at $X_{t}$, of order at most two, with no constant term. Therefore, the
tangent object to $X_{t}$ is formally one of second order. This fact is known as Schwartz principle.

From now on we assume that all {\it semimartingales are continuous}.

Let $X_t$ be a semimartingale in $M$. Let $\Theta_{X_t} \in \tau^*_{X_t}M$ be an adapted stochastic second order form
along $X_t$. Let $(U,x^i)$ be a local coordinate system in $M$. With respect to this chart the second order form
$\Theta$ can be written as \linebreak $\Theta_x= \Theta_i(x)d^2x^i + \Theta_{ij}(x)dx^i\cdot dx^j$, where $\Theta_i$ and
$\Theta_{ij}=\Theta_{ji}$ are ($\mathcal{C}^{\infty}$ say) functions on $M$. Then the integral of $\Theta$ along $X$ is
defined, locally, by
\begin{equation}\label{secondintegral}
\int_0^t \Theta d^2X_{s}=\int_0^t \Theta_i(X_s)dX^i_s + \int_0^t \Theta_{ij}(X_s)d[X^i,X^j]_s.
\end{equation}

Let $b$ be a section of $T^{2}_{0}(M)$, which is defined along $X_t$. The quadratic integral of $b$ along $X_{t}$ is defined, locally,
by
\[
 \int_0^{t}b\;(dX,dX)_s = \int_{0}^{t} b_{ij}(X_{s}) d[X^{i},X^{j}]_{s},
\]
where $b(x) = b_{ij}(x)dx^{i} \otimes dx^{j}$ and $b_{ij}$ are smooth functions. Here, we observe that both integrals $\int_0^t \Theta d^2X_s$ and $\int_0^{t}b\;(dX,dX)_s$ are well defined. To see these facts we refer the reader to \cite{emery1}.

Let $M$ be a smooth manifold endowed with a symmetric connection $\nabla^{M}$. In \cite{meyer},  P. Meyer showed that for $\nabla^{M}$ there exists a section $\Gamma^{M}$ in $Hom(\tau M,TM)$ such that $\Gamma^{M}|_{TM}=Id_{TM}$ and $\Gamma^{M}(AB)= \nabla^{M}_{A}B$, where $A,B \in TM$. We also say connection to $\Gamma^{M}$.

Let $M$ be a smooth manifold endowed with a symmetric connection $\Gamma^{M}$. Let $X$ be a semimartingale in $M$ and $\theta$ be a 1-form along $X$. The It\^o integral is defined by $\int_{0}^{t} \theta d^{M}X:=\int_{0}^{t}\Gamma^{M*}\theta d^{2}X_{t}$. Furthermore, $X$ is said a $\nabla^{M}$-martingale if for every
1-form $\theta$ on $M$ we have that $\int_{0}^{t} \theta d^{M}X$ is a real local martingale.

\begin{definition}
Let $M$ be a Riemmanian manifold with a metric $g$. A semimartingale $B$ in $M$ is called a Brownian motion if $\int_{0}^{t} \theta d^{g}B_t$ is a real local martingale for all $\theta \in T^{*}M$, where $\Gamma^{g}$ is the Levi-Civita connection, and for any section $b$ in $T^{2}_{0}(M)$ we have \begin{equation}\label{Brownian}
\int_0^tb(dB,dB)=\int_0^t\rm{tr}\,b_{B_{s}}ds.
\end{equation}
\end{definition}


\section{Vertical connection}

Let $E, M $ be differential manifolds such that there is a smooth submersion \linebreak $\pi:E\rightarrow M$. Let us
denote the vertical distribution by $VE=\mathrm{ker}\,(\pi_{*})$ and vertical projection by $\mathbf{v}: TE \rightarrow
VE$. A vector field $X$ on $E$ is called vertical if $X_{p} \in V_{p}E$ for all $ p \in E$. Analogous, a 1-form $\theta$ is
called vertical form if $\theta(p) \in V^{*}_{p}E$ for all $p \in E$.

Let $ p \in E$, by Rank Theorem, there exist a coordinate $(x_1, \ldots, x_m, v_{1}, \ldots, v_{k})$ of some neighborhood $U \ni p$ such that
\begin{equation}\label{coordinates}
\pi(x_1,\ldots,x_m,v_{1}, \ldots, v_{k})=  (x_1,\ldots,x_m),
\end{equation}
where $(x_1,\ldots,x_m)$ is a coordinate of a neighborhood $V \ni \pi(p)$. The possibility of choice these coordinates, is fundamental to construct the vertical martingale.

Taking the coordinates (\ref{coordinates}) we obtain in $T_p E$ the coordinate basis \linebreak $\{\bar{D}_{1}(p), \ldots, \bar{D}_{m}(p), D_{1}(p), \ldots, D_{k}(p) \}$, where $\bar{D}_{i} = \partial/\partial x^{i}, i = 1, \ldots m$, and $D_{\alpha}=\partial / \partial v^{\alpha}, \alpha=1, \ldots k$. It is clear that $\{D_{1}(p), \ldots, D_{k}(p)\}$ are vertical vectors in $T_pE$ and it also spans the vertical space $V_pE$. Also, in coordinates (\ref{coordinates}), a
second vector $L$ in $\tau_p E$ is written as
\[
L(p) = a_{\alpha\beta}D_{\alpha\beta}(p) + a_{ij}D_{ij}(p) +a_{\alpha j}D_{\alpha j}(p) + a_{\alpha}D_{\alpha}(p) + a_{i}D_{i}(p)  .
\]
We denote by $\mathfrak{V}_pE$ the subspace spanned by $\{D_{\alpha\beta},D_{\alpha j},D_{\alpha}; \alpha, \beta= 1, \ldots, k \}$. Here, we observe that $\mathfrak{V}_{p}M$  is the ker$\pi_{*}(p)$. We denote by $\mathfrak{V} E = \bigcup_{p \in E}\mathfrak{V}_{p}E$ and by $\mathbf{v}:\tau E \rightarrow \mathfrak{V}E$ the vertical projection.

Given a symmetric connection $\nabla^{E}$ on $E$, in each fiber $\pi^{-1}(x)$, $x \in M$, we can induce a connection $\nabla^{x}$ from the connection $\nabla^{E}$. In this way, $\nabla^{x}$ is the vertical projection of $\nabla^{E}$ on $\pi^{-1}(x)$ for vertical vector fields. For our purpose we generalize this concept in the following way: let $U$ be a vertical vector fields on $E$ and $X$ a vector field on $E$, we define the vertical connection on $E$ by $\nabla^{v}_{X}V= \mathbf{v}\nabla^{E}_{X}V$.

Now we want to see the vertical connection in the context of second order. As $\Gamma^{E}$ is a linear homomorphism from $\tau E$ into $TE$ we can restrict $\Gamma^{E}$ to $\mathfrak{V}E$. A little bit more, we take $\Gamma^{v}(L) = \mathbf{v}\Gamma^{E}(L)$, where $L \in \mathfrak{V}E$. The connection $\Gamma^{v}$ is the object that is associated with $\nabla^{v}$. In fact, if $V,X$ are a vertical vector fields and a vector field on $E$, respectively, then $XV$ is a second vector field in $\mathfrak{V}E$. A simple account gives $\Gamma^{v}(XV) = \nabla^{v}_{X}V$.

Another way to view the vertical connection $\nabla^{v}$ on $E$ is as a horizontal distribution on the frame bundle $BE$ of $E$. We describe in the following this fact.  Denote by $\pi_{BE}$ the natural projection of $BE$ into $E$. We can decompose each tangent space $T_{p}BE$ into the direct sum of the vertical subspace $V_{p}BE = Ker(\pi_{BE*}(p))$ and the horizontal subspace $H_{p}^{v}BE$ of the tangent at $p$ of horizontal lifts of curves in $M$. We recall that if $\gamma : I \rightarrow E$ is a curve in $E$, the horizontal lift of $\gamma$ to $BE$ can be written as the composition
\[
\gamma^{H^{v}}_{p}(t):= P^{\nabla^{v}}_{t,0}(\gamma) \circ p,
\]
where $P^{\nabla^{v}}_{t,s}(\gamma) : T_{\gamma(s)}E \rightarrow T_{\gamma(t)}E$ is the parallel transport along the curve $\gamma$.

\section{Vertical integral of It\^o and Stratonovich}

These three different way to view connection allow us to work with martingales in the following contexts: in second order context introduced by P. Meyer \cite{meyer},  in local coordinates system, founded in \cite{emery1}, and by the context of horizontal lifts and development, founded in \cite{hsu}. Due the nature our work, we choose introduce the vertical martingales through the second order context. However, it is a direct consequence the equivalence with other context, see Proposition \ref{equivalencemartingale}.

We now proceed to construct an integral of It\^o for the vertical connections. Let $X$ be a semimartingale in $E$. Adopting the coordinates (\ref{coordinates}) we obtain, by Schwartz principle,
\[
 d^{2}X_{t}  =  dX^{\alpha}_{t}D_{\alpha} + dX^{i}_{t}D_{i} + \frac{1}{2}d[X^{\alpha},X^{\beta}]_{t}D_{\alpha \beta} +
 \frac{1}{2}d[X^{i},X^{j}]_{t}D_{ij}+ \frac{1}{2}d[X^{i},X^{\beta}]_{t}D_{i\beta}.
 \]
Since that $d^{2}X_{t} \in \tau_{X_t}M$, we can project it into $\mathfrak{V}_{X_{t}}(E)$, that is,
\[
\mathbf{v}(d^{2}X_{t})= dX^{\alpha}_{t}D_{\alpha} + \frac{1}{2}d[X^{\alpha},X^{\beta}]_{t}D_{\alpha \beta}  + \frac{1}{2}d[X^{\alpha},X^{j}]_{t}D_{\alpha j}.
\]

Let $\Theta_{X_t}$ be an adapted stochastic second order form along $X_t$ such that $\Theta_{X_t} \in \mathfrak{V}^*_{X_t}M$. We take the coordinates (\ref{coordinates}) in $E$. With respect to these coordinates the second order form $\Theta$ can be written as
\[
\Theta_p = \Theta_\alpha d^2v^\alpha(p) +  0\Theta_{\alpha\beta}dv^\alpha\cdot dv^\beta(p) + \Theta_{\alpha j}dv^\alpha\cdot dv^j(p),
\]
where $\Theta_\alpha$, $\Theta_{\alpha\beta}=\Theta_{\beta\alpha}$ and $\Theta_{\alpha j}=\Theta_{j \alpha}$ are
($\mathcal{C}^{\infty}$ say) functions in $E$. From definition of the integral (\ref{secondintegral}) we can see that
\begin{eqnarray}\label{verticalintegral}
\!\!\! \int_{0}^t\!\! \Theta d^2X_s \!\!\! & = &\!\!\! \int_0^t\!\! \Theta_\alpha(X_s)dX^\alpha_s \!\! +\!\! \frac{1}{2}\int_0^t\!\!
\Theta_{\beta \gamma}(X_s)d[X^\beta,X^\gamma]_s \! +\! \frac{1}{2}\int_0^t\!\!
\Theta_{\beta j}(X_s)d[X^\beta,X^j]_s \nonumber \\
& = &\!\!\! \int_{0}^t \Theta \mathbf{v} d^2X_t.
\end{eqnarray}

We observed that the integral $\int_{0}^t \Theta \mathbf{v} d^2X_t$ is well posed. In the sense that it is based in the
$\int_{0}^t \Theta d^2X_t$, which has a good definition (see for example Theorem 2.10 in \cite{emery2} or section 4 in
\cite{meyer}).

Now, since $\mathbf{v} d^{2}X_{t} \in \mathfrak{V}_{X_{t}}(E)$, there is sense in $\Gamma^{v}(\mathbf{v}(d^{2}X_{t}))$. A little bit more, if $\theta$ is a vertical form in $E$, taking the coordinates (\ref{coordinates}) we can write $\theta(p) = \theta_{\alpha}(p) dv^{\alpha}$ and, consequently,
\[
\Gamma^{v*}\theta(p)\! =\! \theta^{\alpha}(p)(d^2v^{\alpha}\! +\! \Gamma^{\alpha}_{\beta \gamma}(p) dv^{\beta}\cdot dv^{\gamma}\! +\! \Gamma^{\alpha}_{\beta j}(p) dv^{\beta}\cdot dv^{j}).
\]
It follows that every tools to define a It\^o integral for vertical connections are well posed.

\begin{definition}
Let $E, M$ be differential manifolds such that there is a smooth submersion $\pi:E\rightarrow M$. Let $\nabla^{E}$ be a
symmetric connection, $X$ a semimartingale on $E$ and $\theta$ a vertical form on $E$. We define the vertical It\^o integral of $\theta$ along $X$ as $\int_{0}^{t} \theta d^{v} X_s = \int_{0}^{t} \Gamma^{v*}\theta (\mathbf{v}d^{2}X_{s})$. Let $(U,x_1,\ldots,x_m, v_{1}, \ldots, v_{k})$ be a chart such that
(\ref{coordinates}) is true. Then, locally, the vertical It\^o integral of $\theta$ along $X$ is given by
\[
\int_{0}^{t}\!\! \theta d^{v} X_s\! =\! \int_0^t\!\! \theta_\alpha(X_s)dX^\alpha_s \!+\! \frac{1}{2}\!\int_0^t\!\!
\Gamma^\alpha_{\beta \gamma}\theta_\alpha(X_s)d[X^\alpha,X^\beta]_s \!+\! \frac{1}{2}\!\int_0^t\!\!
\Gamma^\alpha_{\beta j}\theta_\alpha(X_s)d[X^\alpha,X^j]_s.
\]
\end{definition}

One can observe that the vertical It\^o integral is well posed because the right side of the above equality is an
integral as (\ref{verticalintegral}).

From the Definition above we can define our main object: the vertical martingales.

\begin{definition}
A semimartingale $X_{t}$ in $E$ is called a vertical martingale if $\int_{0}^{t} \theta d^{v} X_s $ is a real local martingale for every
vertical form $\theta$ on $E$.
\end{definition}

Let $X$ be a semimartingale on $E$ and  $W$ the anti-development of $X$ in $\mathbb{R}^{m+k}$, see for example section 2.3 in \cite{hsu}. Write  $W= W_{1} + W_{2}$, where $W_{1}$ is a semimartingale in $\mathbb{R}^{m}$ and $W_{2}$ is a semimartingale in $\mathbb{R}^{k}$. A stopping time argument assures that we can view a semimartingale $X$ in $E$ in a local coordinates system. This allows to characterize the vertical martingales in the following way.

\begin{prop}\label{equivalencemartingale}
Suppose that $(U,x_1,\ldots,x_m, v_{1}, \ldots, v_{k})$ be a chart on $E$ such that (\ref{coordinates}) is true. Then we are the equivalences:
\begin{enumerate}
 \item $X$ is a vertical martingale.
 \item  $ X^{\alpha}(t)= X^{\alpha}(0) +\, local\, martingale\, +  \frac{1}{2}\int_{0}^{t} \Gamma^{\alpha}_{\beta R}(X_{t})d[X^{\beta}, X^{R}], \alpha,\beta=1, \ldots k$ and $R=1,\ldots, k+m$.
 \item $W_{2}$ is local martingale in $\mathbb{R}^{k}$.
\end{enumerate}
\end{prop}
\begin{proof}
Let $(U,x_1,\ldots,x_m, v_{1}, \ldots, v_{k})$ be a chart on $E$ such that (\ref{coordinates}) is true. Then $X^{\alpha} = v_{\alpha} \circ X$ and $X^{j} = x_{j} \circ X$. It is direct to see that $dv^{\alpha}$, $\alpha= 1, \ldots, k$, are vertical forms. Let $M$ a $\mathbb{R}^{k+m}$-valued semimartingale defined by
\[
M^{\alpha}(t) = X^{\alpha}(t) - X^{\alpha}(0) -  \frac{1}{2}\int_{0}^{t} {\Gamma^{v}}^{\alpha}_{\beta R}(X_{t})d[X^{\beta}, X^{R}]
\]
and
\[
 M^{j}(t) = X^{j}(t) - X^{j}(0) -  \frac{1}{2}\int_{0}^{t} {\Gamma^{v}}^{j}_{Sl}(X_{t})d[X^{S}, X^{l}],
\]
where $R,S = 1, \ldots, m+k$ and ${\Gamma^{v}}^{\alpha}_{\beta R}$ and ${\Gamma^{v}}^{j}_{Sl}$ are the Christoffel symbols of $\nabla^{v}$. By definition of vertical martingale, we are only interested in $M^{\alpha}(t)$. In fact, $M^{\alpha}(t) = \int_{0}^{t} dv^{\alpha}d^{v}X_{s}$. Therefore, the proof is a direct adaptation of proof of Proposition 2.5.4 in \cite{hsu} and characterization of martingales in local coordinates founded at page 37 in \cite{emery1}.\qed
\end{proof}



Our next step is to define a vertical Stratonovich integral. Let $U$ be a chart in $E$  with coordinates (\ref{coordinates}). Let $\theta$ be a vertical form on $E$, so in coordinates we have $\theta = \theta_{\alpha} dv^{\alpha}$. It is well known to define the Stratonovich integral we need to yields a second order form from one form $\theta$. For such, we use the operator $d:T^{*}E \rightarrow \tau^{*}E$ which is, locally, given by
\[
d\theta(p) = \theta_{\alpha}(p) d^2v^{\alpha} + D_{\alpha}\theta_{\beta}(p)dv^{\alpha}\cdot dv^{\beta}+
D_{i}\theta_{\alpha}(p)dx^{i}\cdot dv^{\alpha}.
\]
It is direct to see that $d\theta$ belongs to $\mathfrak{V}^{*}E$. Therefore, taking a semimartingale $X$ in $E$ there is sense in
\[
d\theta(p)(\mathbf{v}d^{2}X_{t}) = \theta_{\alpha}(X_{t})dX^{\alpha} +
\frac{1}{2}D_{\alpha}\theta_{\beta}(X_{t})d[X^{\alpha},X^{\beta}] + +
\frac{1}{2}D_{i}\theta_{\beta}(X_{t})d[X^{i},X^{\beta}].
\]
Therefore, for (\ref{verticalintegral}), we have a good definition for the vertical Stratonovich integral.

\begin{definition}
Let $E, M$ be differential manifolds such that there is a smooth submersion $\pi:E\rightarrow M$. Let $X$ be a semimartingale on $E$ and $\theta$ a vertical form on $E$ along $X$. We define the vertical Stratonovich integral of $\theta$ along $X$ as $\int \theta \delta^{v} X_t = \int d\theta (\mathbf{v}d^{2}X_{t})$. Let $(U,x_1,\ldots,x_m, v_{1}, \ldots, v_{k})$ be a chart such that (\ref{coordinates}) is true. Then, locally, the vertical Stratonovich integral of $\theta$ along $X$ is given by
\[
\int_0^t\!\! \theta \delta^{v} X_s\! =\! \int_0^t\!\! \theta_\alpha(X_s)dX^\alpha_s\! +\! \frac{1}{2}\int_0^t\!\!\!
D_\alpha\theta_\beta(X_s)d[X^\alpha,X^\beta]_s\! +\! \frac{1}{2}\int_0^t\!\!\!
D_i \theta_\beta(X_s)d[X^i,X^\beta]_s.
\]
\end{definition}

From definition of the vertical It\^o and Stratonovich integrals we can show a change formula between they. In fact, a simple computation in coordinates gives

\begin{prop}\label{changeverticalformula}
For a vertical form $\theta$ on $E$ and a semimartingale $X_t$ on $E$ we have
\[
\int^{t}_{0} \theta \delta^{v}X_{s} = \int^{t}_{0} \theta d^{v}X_{s} - \frac{1}{2}\int^{t}_{0} \nabla^{v}\theta(dX,\mathbf{v}dX)_s,
\]
where $\mathbf{v}dX$ is the vertical projection of $dX$ into $VE$.
\end{prop}

\section{The geometric It\^o formulas}

Our main goal in this section is to construct some tools of the Stochastic Calculus, in a vertical way. More explicitly, we will construct the It\^o formulas for vertical It\^o integral and vertical Stratonovich integral. Our ideas are inspired in \cite{catuogno2}. For beginning, we show a version of Proposition 1.8 in \cite{emery2}.

\begin{lema}\label{commuteverticalform}
Let $E, M$ be differential manifolds such that there is a smooth submersion $\pi:E\rightarrow M$, $\phi:N \rightarrow
E$ a smooth map and $\theta$ a vertical form on $E$. Then
\[
(\mathbf{v}\phi)^{*}d\theta = d((\mathbf{v}\phi)^{*}\theta).
\]
\end{lema}
\begin{proof}
We first adopt the coordinates (\ref{coordinates}). Since $\theta$ is a vertical form, follwos that $(\mathbf{v}\phi)^{*}\theta = \phi^{*}\theta$. It is direct that $(\phi^{*}\theta) = (\theta^{\alpha} \circ \phi) d(v^{\alpha} \circ \phi)$. Applying the operator $d$ at $(\phi^{*}\theta)$ we compute
\begin{eqnarray*}
d(\phi^{*}\theta) & = & d((\theta^{\alpha}\circ \phi)d(v^{\alpha}\circ \phi))\\
& = & d(\theta^{\alpha} \circ \phi) d(v^{\alpha} \circ \phi) + (\theta^{\alpha} \circ \phi) d^{2}(v^{\alpha} \circ \phi) \\
& = & \phi^{*}d\theta^{\alpha}\phi^{*}dv^{\alpha} + (\theta^{\alpha} \circ \phi)\phi^{*}d^{2}v^{\alpha}\\
& = & \phi^{*}(d\theta^{\alpha}dv^{\alpha} + \theta^{\alpha}d^{2}v^{\alpha})
\end{eqnarray*}
Now the differential of $\theta^{\alpha}$ is given by
\[
d\theta^{\alpha} = \frac{\partial\theta^{\alpha}}{\partial x_{i}}dx^{i} + \frac{\partial\theta^{\alpha}}{\partial v_{\beta}}dv^{\beta}.
\]
We thus obtain
\[
d(\phi^{*}\theta)  =  \phi^{*}(\frac{\partial\theta^{\alpha}}{\partial x_{i}}dx^{i}dv^{\alpha}+ \frac{\partial \theta^{\alpha}}{\partial v^{\beta}}dv^{\beta}dv^{\alpha} + \theta^{\alpha} d^{2}v^{\alpha}).
\]
Since the right side, in coordinates, is $\phi^{*}d\theta$ and $\phi^{*}d\theta = (\mathbf{v}\phi)^{*}d\theta$, because $d\theta \in \mathfrak{V}^{*}E$, we conclude that
\[
d((\mathbf{v}\phi)^{*}\theta) = (\mathbf{v}\phi)^{*}d\theta.
\]
\qed
\end{proof}

The result of Lemma above for any first order form is responsible to the It\^o formula for Stratonovich integral (see for example \cite{emery2} or \cite{meyer}). Following the same idea, Lemma above was constructed specifically to show

\begin{theorem}
Let $E, M$ be differential manifolds such that there is a smooth submersion $\pi:E\rightarrow M$, $\phi:N \rightarrow
E$ a smooth map and $\theta$ a vertical form on $E$. Then
\begin{equation}\label{itostratonovichformula}
\int_{0}^{t} \theta \delta^{v}\phi(X_s) = \int_{0}^{t} (\mathbf{v}\phi^{*}\theta) \delta X_s.
\end{equation}
\end{theorem}
\begin{proof}
Let $X_t$ be a semimartingale in $N$ and $\theta$ a vertical form on $E$. By definition of the vertical Stratonovich integral,
\begin{eqnarray*}
\int_{0}^{t} \theta \delta^{v}\phi(X_{s})
& = & \int_{0}^{t} d\theta \mathbf{v} d^{2}\phi(X_{s}) = \int_{0}^{t} (\mathbf{v}\phi)^{*}d\theta d^{2}X_{s} =  \int_{0}^{t} d((\mathbf{v}\phi)^{*}\theta) d^{2}X_{s}\\
& = & \int_{0}^{t} (\mathbf{v}\phi)^{*}\theta \delta X_{s},
\end{eqnarray*}
where we used Lemma \ref{commuteverticalform} in the third equality.\qed
\end{proof}

Beyond the It\^o formula for Stratonovich integral the stochastic calculus in manifolds has another formula for transformation between manifolds: the geometric It\^o formula (see for example \cite{catuogno2}). Our next purpose is to construct the geometric It\^o formula for the vertical It\^o integral. We begin introducing the second fundamental form, tension field and vertical harmonic map.

\begin{definition}\label{defsecondformvertical}
Let $E, M$ be differential manifolds such that there is a smooth submersion $\pi:E\rightarrow M$ and $\phi:N \rightarrow E$ a smooth map. Suppose that $E$ and $N$ are equipped with symmetric connections $\nabla^{E}$ and $\nabla^{N}$, respectively. Furthermore, denote the vertical connection by $\Gamma^{v}$. The section $\alpha^{v}_{\phi}$ of $\mathfrak{V}^{*}E \otimes \phi^{*}VE$ is given by
\begin{equation}\label{secondformvertical}
 \alpha^{v}_{\phi} = \Gamma^{v} \mathbf{v}\phi_{*} - \mathbf{v}\phi_{*} \Gamma^{N}.
\end{equation}
The vertical second  fundamental form of $\phi$, $\beta_{\phi}^{v}$, is the unique section of \linebreak $ (TM \odot TM)^*\otimes \phi^{*}VE$ such that $\alpha_{\phi}^{v} = \beta^{v}_{\phi}\circ Q$. The tension field of $\beta^{v}_{\phi}$ is
\[
 \tau_{\phi}^{v} = \rm{tr} \beta^{v}_{\phi}.
\]
We call $\phi$ a vertical harmonic map if $\tau_{\phi}^{v}=0$.
\end{definition}

The following linear algebra lemma shows that $\beta_{F}^{v}$ is well defined.

\begin{lema}\label{lemma1}
Let $\alpha^{v}_{\phi}$ be  a section of $\mathfrak{V}^*E \otimes \phi^{*}VE$ defined by (\ref{secondformvertical}). Then there exists an unique section $\beta^{v}_{\phi}$ of $(TM\odot TM)^*\otimes \phi^{*}VE$ such that $\alpha_{\phi}^{v}=\beta_{\phi}^{v} \circ Q$.
\end{lema}
\begin{proof}
Since $\mathrm{Ker}\,Q =TM \subset \mathrm{Ker}\,\alpha_{\phi}^{v}$, the lemma follows from the first isomorphism theorem (see \cite{Rot} pp 67). \qed
\end{proof}

The following lemma is necessary in the main Theorem of this section.

\begin{lema}\label{lemma2}
Under assumptions in Definition \ref{defsecondformvertical}, for each vertical form $\theta$ on $E$,
\[
\int_{0}^{t} \alpha^{v*}_{\sigma} \theta~ d_2X_s = \frac{1}{2}\int_{0}^{t} \beta^{v*}_{\sigma} \theta (dX,dX)_s.
\]
\end{lema}
\begin{proof}
By definition of $\beta_{\phi}^v$, for each vertical form $\theta$ we have
\[
\frac{1}{2}\int_{0}^{t} \beta^{v*}_{\sigma} \theta (dX,dX)_s=\int_{0}^{t} Q^{*}\beta^{v*}_{\sigma} \theta~
d^{2}X_s = \int_{0}^{t} (\beta_{\sigma}^v \circ Q)^{*} \theta~ d^{2}X_s = \int_{0}^{t} \alpha^{v*}_{\sigma}\theta~ d^{2}X_s.
\]
The first equality follows from Proposition 6.31 in \cite{emery1}.\qed
\end{proof}

Let us make a clear observation. Lemma \ref{lemma1} assures an existence $\beta_{\phi}^{v}$, in the other hand, Lemma
\ref{lemma2} shows as related the integral of $\alpha_{\phi}^{v}$ and $\beta_{\phi}^{v}$. Both results are vital for
the construction of the following geometric It\^o formula.

\begin{theorem}\label{geometricitoformula}
Let $E, M,N$ be differential manifolds such that there is a smooth submersion $\pi:E\rightarrow M$ and $\phi:N \rightarrow E$ a smooth map. Suppose that $E$ and $N$ are equipped with symmetric connections $\nabla^{E}$ and $\nabla^{N}$, respectively. Furthermore, denote the vertical connection by $\nabla^{v}$. If $X$ is an $N$-valued semimartingale and $\theta$ is a vertical form on $E$,
then
\[
\int_{0}^{t} \theta d^{v}\phi(X_s)=\int_{0}^{t} \phi^*\theta d^{N}X_s+\frac{1}{2}\int_{0}^{t} \beta_\phi^{v*}\theta(dX,dX)_s.
\]
\end{theorem}
\begin{proof} We calculate
\[
\begin{array}{rcl}
\int_{0}^{t} \theta d^{v}\phi(X_s) & = & \int_{0}^{t} (\Gamma^{v*} \theta) \mathbf{v} d^2\phi(X_s)\\
& = &   \int_{0}^{t} (\Gamma^{v*} \theta) \mathbf{v}\phi_* d^2(X_s) = \int_{0}^{t} \mathbf{v}\phi^*(\Gamma^{v*}\theta) d^2X_s \\
& = & \int_{0}^{t} \mathbf{v}\phi^*(\Gamma^{v*}\theta) d^2X_s + \int_{0}^{t} \Gamma^{N*}(\phi^*
\theta) d^2X_s - \int_{0}^{t} \Gamma^{N*}(\phi^* \theta) d^2X_s \\
& = & \int_{0}^{t} \Gamma^{N*}(\phi^* \theta) d^2X_s + \int \Big(
\mathbf{v}\phi^*(\Gamma^{v*} \theta) - \Gamma^{N*}(\phi^* \theta) \Big) d^2X_s \\
& = & \int_{0}^{t} \phi^* \theta d^{N}X_s + \int \alpha^{v*}_\phi \theta d_2X_s \\
& = & \int_{0}^{t} \phi^* \theta d^{N}X_s + \frac{1}{2}\int_{0}^{t} \beta^{v*}_\phi \theta (dX,dX)_s,
\end{array}
\]
where we use Lemma \ref{lemma2} in the last equality.
\end{proof}

\begin{corollary}
Under assumptions of Theorem \ref{geometricitoformula}, furthermore, $(N,g)$ is a Riemannian and $\nabla^{g}$ is the
Levi-Civita connection on $N$. If $B$ is a $g$-Brownian motion in $N$ and $\theta$ be a vertical form on $E$, then
\begin{equation}\label{harmonic}
\int_{0}^{t} \theta d^{v}\phi(B_s)=\int_{0}^{t} \phi^*\theta d^{g}B_s+\frac{1}{2}\int_{0}^{t} \tau_\phi^{v*}\theta(B_s)ds.
\end{equation}
\end{corollary}

A link between Stochastic Analysis and Differential Geometry is a well known stochastic characterization of harmonic maps, see for example \cite{emery1} or \cite{emery2}. One can observe that until here we construct analogous tools to show a stochastic characterization for vertical harmonic maps.

\begin{prop}\label{verticalharmoniccharact}
Let $E, M$ be differential manifolds such that there is a smooth submersion $\pi:E\rightarrow M$. Suppose that $E$ is equipped with a symmetric connection $\nabla^{E}$ and $(N,g)$ is a Riemannian manifold. Denote the vertical connection by $\nabla^{v}$. A smooth map $\phi:N \rightarrow E$ is vertical harmonic map if and only if $\phi$ sends $g$-Brownian motions $B_t$ in vertical martingales $\phi(B_t)$.
\end{prop}
\begin{proof}
Let $B_t$ be a $g$-Brownian motion in $N$ and $\theta$ a vertical form on $E$. By formula (\ref{harmonic}),
\[
\int_{0}^{t} \theta~ d^{v}\phi(B_{s}) = \int_{0}^{t} \phi^{*} \theta d^{g}B_{s} + \frac{1}{2}\int_{0}^{t} \tau_{\phi}^{v*}\theta(B_{s})ds.\\
\]
We observe that $\int_{0}^{t} \phi^{*} \theta~d^{g}B_{s}$ is a real local martingale. Since $B_t$ and $\theta$ are arbitrary, the Doob-Meyer decomposition assures that $\int_{0}^{t} \theta~ d^{v}\phi(B_{s})$ is a real local martingale if and only if $\tau_{\phi}^{v}$ vanishes. From the definitions of vertical martingales and vertical harmonic maps we conclude the proof. \qed
\end{proof}

\section{Harmonic section}

In this section, we work with object that motivate this study: harmonic section. Before going further, we give the environment for the study of the harmonic sections.

Let $E$ be a differential manifold and $(M,g)$ a Riemannian manifold such that there is a smooth submersion $\pi:E\rightarrow M$. Let $\nabla^{E}$ be a symmetric connection on $E$ such that $\pi$ has totally geodesics fibers property and denote by $\nabla^{g}$ the Levi-Civita connection on $M$. Let $VE=\mathrm{ker}\,(\pi_{*})$ be the vertical distribution and $HE$ a smooth distribution in $TE$ such that $TE= VE\oplus HE$. Let $\mathbf{v}: TE \rightarrow VE$ and $\mathbf{h}:TE \rightarrow HE$ be the vertical and horizontal projectors, respectively. Let \linebreak $H_{p} = (\pi_{p∗}|_{H_{p}E})^{-1} : T_{x}M \rightarrow H_{p}E$ be the horizontal lift, where $\pi(p)=x$. We observe that $H_{p}$ is an isomorphism for each $p \in E$. The submersion \linebreak $\pi:E \rightarrow M$ is called affine submersion with horizontal distribution if  \linebreak $\mathbf{h}\nabla^{E}_{H(X)}H(Y) = H(\nabla^{M}_{X}Y)$ (see
\cite{abe} for more details). A Riemmanian submersion is a classical example of affine submersion with horizontal distribution.

In this section, unless otherwise stated, we assume that $\pi:E \rightarrow M$ is an affine submersion with horizontal distribution.

Next we extend the definition given by C. M. Wood \cite{wood1} for harmonic sections.

\begin{definition}
A section $\sigma$ of $\pi$ is called a harmonic section if $\tau_{\sigma}^{v} =0$.
\end{definition}

An immediate consequence of Proposition \ref{verticalharmoniccharact} is the characterization of the harmonic sections in the following stochastic context.

\begin{theorem}\label{harmonicsection}
Let $E$ be a differential manifold and $(M,g)$ a Riemannian manifold such that there is a smooth submersion $\pi:E\rightarrow M$. Assume that $E$ is equipped with a symmetric connection $\nabla^{E}$ such that $\pi$ has totally geodesics fibers property. Then a section $\sigma$ of $\pi$ is harmonic section if and only if, for every $g$-Brownian motion $B_t$ in $M$, $\sigma(B_t)$ is a vertical martingale in $E$.
\end{theorem}

\section{Applications}

\noindent{{\large{\it Tangent Bundle with Complete lift}}\\

Let $(M,g)$ be a Riemannian manifold and $TM$ its tangent bundle. To study the vertical martingales in $TM$ we need to introduce a connection on it. Denoting by $\nabla^{g}$ the Levi-Civita connection we prolong $\nabla^{g}$ to the complete lift $\nabla^{c}$ on $TM$ (see \cite{yano} for the definition of $\nabla^{c}$). In a nutshell, if $X,Y$ are vector fields on $M$, then $\nabla^{c}$ satisfies the following equations:
\begin{equation}\label{canonicallift}
\begin{array}{ccl}
\displaystyle \nabla^{c}_{X^{v}}Y^{v} & = & 0  \\
\nabla^{c}_{X^{v}}Y^{h} & = & 0  \\
\nabla^{c}_{X^{h}}Y^{v} & = & (\nabla_{X}Y)^{v} \\
\nabla^{c}_{X^{h}}Y^{h} & = & (\nabla_{X}Y)^{h} + \gamma(R(-,X)Y,
\end{array}
\end{equation}
where $R(-,X)Y$ denotes a tensor field $W$ of type (1,1) on $M$ such that \linebreak $W(Z)=R(Z,X)Y$ for any $Z \in T^{(0,1)}(M)$, and $\gamma$ is a lift of tensors, which is defined at page 12 in \cite{yano}. Furthermore, $X^{v},Y^{v}$ and $X^{h},Y^{h}$ are the vertical and horizontal lift of the $X,Y$ on $TM$, respectively. A direct account shows that $\pi_{TM}: TM \rightarrow M$ is an affine submersion with horizontal distribution.

\begin{prop}\label{verticalcanonical}
Let $(M,g)$ be a Riemannian manifold and $TM$ its tangent bundle equipped with the complete lift $\nabla^{c}$. Let $X_{t}$ be a semimartingale and $J_{t} = \pi_{TM}(X_{t})$ a semimartingale in $M$. Then $X_t$ is a vertical martingale if and only if, for each vertical form $\theta$,
\[
\int_{0}^{t} \theta \delta^{v} X_{s} - \int_{0}^{t} \theta (\delta J_{s})^{v} + \int_{0}^{t} \theta^{v*} d^M J_{s}
\]
is a real local martingale, where $\theta^{v*}$ is the push-forward of $\theta$ by the vertical lift on $TM$. In the case that $J_{t}$ is a $\nabla^{M}$-martingale, $X_{t}$ is a vertical martingale if and only if
\[
\delta^{v}X_{t} = (\delta J_{t})^{v}.
\]
\end{prop}
\begin{proof}
Let $X_{t}$ be a semimartingale in $E$ and $\theta$ a vertical form on $E$. Then from change formula \ref{changeverticalformula} we see that
\[
\int_{0}^{t} \theta d^{v}X_{s} = \int_{0}^{t} \theta \delta^{v} X_{s} + \frac{1}{2}\int_{0}^{t} \nabla^{v}\theta(dX,dX)_{s}.
\]
We now calculate $\nabla^{v}\theta$. Taking a vector field $A$ on $E$ and denoting $B = \pi_{TM*}(A)$ we compute
\begin{eqnarray*}
\nabla^{v}\theta(A,\mathbf{v}A)
& = & A\theta(\mathbf{v}A) - \theta(\nabla^{v}_{A}\mathbf{v}A)\\
& = & \mathbf{v}A\theta(\mathbf{v}A) - \theta(\nabla^{v}_{\mathbf{h}A}\mathbf{v}A),
\end{eqnarray*}
where we used that $\nabla^{v}_{\mathbf{v}A}\mathbf{v}A = 0$ in the last equality. As $B = \pi_{TM*}(A)$ we have $\mathbf{v}A = B^{v}$ and $\mathbf{h}A = B^{h}$. We thus obtain
\[
\nabla^{v}\theta(A,\mathbf{v}A) = B^{v}\theta(B^v) - \theta(\nabla^{M}_{B}B)^v,
\]
due to (\ref{canonicallift})  and definition of $\nabla^v$ on $TM$. Denoting the push-forward of $\theta$ by vertical lift as $\theta^{v*}$ we deduce that
\[
\nabla^{v}\theta(A,\mathbf{v}A) = B(\theta^{*v}B) - \theta^{v*}(\nabla^{M}_{B}B) = \nabla^{M}\theta^{v*}(B,B).
\]
Denote $J_t = \pi_{TM*}(X_t)$. It follows that
\[
\int_{0}^{t} \theta d^{v}X_{s} = \int_{0}^{t} \theta \delta^{v} X_{s} + \frac{1}{2}\int_{0}^{t} \nabla^{M}\theta^{v*}(dJ,dJ)_{s}.
\]
We again use the change formula to obtain
\[
\int_{0}^{t} \theta d^{v}X_{s} = \int_{0}^{t} \theta \delta^{v} X_{s} - \int_{0}^{t} \theta^{v*} \delta J_{s} + \int_{0}^{t} \theta^{v*}d^M J_{s},
\]
and the proof follows. \qed
\end{proof}

Our next step is to use the Proposition \ref{verticalcanonical} and Theorem \ref{harmonicsection} to study the harmonic section of $\pi_{TM}$.

\begin{prop}
Let $(M,g)$ be a Riemannian manifold and $TM$ its tangent bundle equipped with the complete lift $\nabla^{c}$.  Then a section $\sigma$ of $\pi_{TM}$ is harmonic section if and only if $\mathbf{v}\sigma_{*}$ is null.
\end{prop}
\begin{proof}
If $\mathbf{v}\sigma_{*}$ is null, it is immediate that $\sigma$ is harmonic map. Suppose, contrary to our claim, that there exist a harmonic section $\sigma$ of $\pi_{TM}$ such that $\mathbf{v}\sigma_{*}$ is no null. Since $\sigma$ is a harmonic section, Theorem \ref{harmonicsection} assures that for every Brownian motion $B_{t}$ in $M$ we have that $\sigma(B_{t})$ is a vertical martingale. Applying this in Proposition \ref{verticalcanonical} and using the It\^o formula for vertical Stratonovich integral (\ref{itostratonovichformula}) we get
\[
\mathbf{v}\sigma_{*}\delta B_{s}= (\delta B_{s})^{v}.
\]
In this way, taking two harmonic sections $\sigma_{1}$ and $\sigma_{2}$ we have $\mathbf{v}\sigma_{1*}\delta B_{s}=\mathbf{v}\sigma_{2*}\delta B_{s}$, for any Brownian motion $B_{t}$ in $M$. However, if $\sigma_{1}$ is harmonic section, then $\sigma_{2} = 2 \sigma_{1}$ is too. So we thus get $\mathbf{v}\sigma_{1*}\delta B_{s}=2 \mathbf{v}\sigma_{1*}\delta B_{s}$, a contradiction. \qed\\
\end{proof}


{\large {\it Tangent bundle with Sasaki metric}}\\

Let $M$ be a Riemannian manifold and $TM$ the tangent bundle equipped with the Sasaki metric $g_{s}$. See for example \cite{gudmundsson} for a complete study about Sasaky metric. In therms of vertical and horizontal distribution the Levi-Civita connection $\nabla^{s}$ for Sasaky metric is given by
In a nutshell, if $X,Y$ are vector fields on $M$, then $\nabla^{c}$ satisfies the following equations:
\begin{equation}\label{canonicallift}
\begin{array}{ccl}
\displaystyle \nabla^{s}_{X^{v}}Y^{v} & = & 0  \\
\nabla^{s}_{X^{v}}Y^{h} & = & \frac{1}{2}(R(-,X)Y)^{h} \\
\nabla^{s}_{X^{h}}Y^{v} & = & (\nabla_{X}Y)^{v}+ \frac{1}{2}(R(-,X)Y)^{h} \\
\nabla^{s}_{X^{h}}Y^{h} & = & (\nabla_{X}Y)^{h} - \frac{1}{2}(R(-,X)Y)^{v},
\end{array}
\end{equation}
where $R(-,X)Y$ denotes a tensor field $W$ of type (1,1) on $M$ such that \linebreak $W(Z)=R(Z,X)Y$ for any $Z \in T^{(0,1)}(M)$. A simple observation shows that the vertical connection on $TM$ defined from $\nabla^{s}$ is equal to vertical connection defined from $\nabla^{c}$. Furthermore, $\pi_{TM}: TM \rightarrow M$ is an affine submersion with horizontal distribution. Then, similarly, we have the following results.

\begin{prop}
Let $(M,g)$ be a Riemannian manifold and $TM$ its tangent bundle equipped with the Sasaky metric. Let $X_{t}$ be a semimartingale and \linebreak $J_{t} = \pi_{TM}(X_{t})$ a semimartingale in $M$. Then $X_t$ is a vertical martingale if and only if, for each vertical form $\theta$,
\[
\int_{0}^{t} \theta \delta^{v} X_{s} - \int_{0}^{t} \theta (\delta J_{s})^{v} + \int_{0}^{t} \theta^{v*} d^M J_{s}
\]
is a real local martingale, where $\theta^{v*}$ is the push-forward of $\theta$ by the vertical lift on $TM$. In the case that $J_{t}$ is a $\nabla^{M}$-martingale, $X_{t}$ is a vertical martingale if and only if
\[
\delta^{v}X_{t} = (\delta J_{t})^{v}.
\]
\end{prop}

\begin{prop}
Let $(M,g)$ be a Riemannian manifold and $TM$ its tangent bundle equipped with the Sasaky metric. Then a section $\sigma$ of $\pi_{TM}$ is harmonic section if and only if $\mathbf{v}\sigma_{*}$ is null.
\end{prop}


{\large {\it Principal Riemannian fiber bundle}}\\

M. Arnaudon and S. Paycha, in \cite{arnaudon-paycha}, shows that semimartingales in a principal fiber bundle $P(M,G)$
with $G$-invariant Riemannian metric $k$ can be decomposed into $G$- and $M$- valued semimartingales. More precisely, a
semimartingale $X$ with values in $P(M,G)$ splits in a unique way into a horizontal semimartingale $X^h$ and a
semimartingale $V$ with values in $G$ such that
\[
X= X^h \cdot V.
\]
Moreover, $V$ is the stochastic exponential
\[
V = \epsilon(\int \omega \delta X)
\]
and $X^h$ is the solution of the It\^o equation
\[
d^{\nabla^k}X^h=H^k_{\widetilde{X}}d^{\nabla}(\pi \circ X).
\]

Now we induce from the metric $k$, at each fiber $\pi^{-1}(x)$, $x \in M$, a metric $k_{x}$ such that $\pi_{P}:P \rightarrow M$ has totally geodesic fibers property. Also we will induce a metric $h$ in $G$ such that $p: G \rightarrow \pi^{-1}(x)\subset P$ is an isometric map, for each $p \in P$. Let us denote by $\nabla^{k},\nabla^{x}$ and $\nabla^{h}$ the Levi-civita connections on $P,\pi^{-1}(x)$ and $G$, respectively. Observing that $\nabla^{k}$ induces a connection $\nabla^{x}$ at each fiber and the vertical connection $\nabla^{v}$ coincides with latter in the fibers we conclude that $p$ is affine maps according $\nabla^{h}$ and $\nabla^{v}$.

Our finally goal is to study the vertical martingales in the environment described above. We begin with necessary Lemma.

\begin{lema}
If $\theta$ is a vertical form on $P(M,G)$ and if $A$ is a vector field on $P(M,G)$, then at each $p \in P$
\[
\nabla^{v}\theta(A,\mathbf{v}A)_{p} = \nabla^{G}u^{*}\theta(\xi,\xi)_{g},
\]
where $A = A^{h} + \xi$,  with $A^{h}$ is the horizontal lift of the $\pi_{*}(A)$ and $\xi$ a vertical vector field, and $p = u\cdot g$, with $u \in \pi_{P}^{-1}(x)$, $x = \pi_{P}(p)$, and $g \in G$.
\end{lema}
\begin{proof}
Let $A$ be a vector fields. It is possible to write $A$ as $A = A^{h} + \xi$, where $\xi$ is a vector field on $G$ and $A^{h}$ is the horizontal lift of $\pi_{*}(A)$. At each $p \in P$ we compute
\begin{eqnarray*}
\nabla^{v*}\theta(A,\mathbf{v}A)_{p}
& = & A_{p}\theta(\mathbf{v}A)_{p} - \theta_{p}(\nabla^{v}_{A}\mathbf{v}A(p)) \\
& = & (A^{h}+ \xi)_{p}\theta(\xi)_{p} - \theta_{p}(\nabla^{v}_{A^{h} + \xi}\xi(p)) \\
& = & \xi_{p}\theta(\xi)_{p} - \theta_{p}(\nabla^{v}_{A^{h}}\xi(p) + \nabla^{v}_{\xi}\xi(p)).
\end{eqnarray*}
Since $\nabla^{k}$ is free torsion, $\pi_{P}$ has has totally geodesic fibers property and $\theta$ is vertical form, it follows that $\nabla^{v}_{A^{h}}\xi(p) = 0$. Therefore
\[
\nabla^{v*}\theta(A,\mathbf{v}A)_{p} = \xi_{p}\theta(\xi)_{p} - \theta_{p}(\nabla^{v}_{\xi}\xi(p)).
\]
Write $p = u\cdot g$. Because each $u \in P$ is a diffeomorphism and an affine map from fiber $\pi^{-1}(x)$ and $G$, $x =\pi_{P}(u)$, we conclude that
\begin{eqnarray*}
\nabla^{v*}\theta(A,\mathbf{v}A)_{p}
& = & u_{*}\xi_{g}u^{*}\theta(\xi_{g}) - u^{*}\theta(\nabla^{G}_{u_{*}\xi}u_{*}\xi(g))\\
& = & \nabla^{G}(u^{*}\theta)(u_{*}\xi,u_{*}\xi)(g),
\end{eqnarray*}
and the proof is complete.\qed
\end{proof}

The point of the Lemma is that it allows one to see that the dual connection on the fibers is the one on Lie group. It is the fundamental key to prove the

\begin{prop}
Let $X$ be a semimartingale in $P(M,G)$. $X$ is a vertical martingale if and only if $V$ is a $\nabla^{G}$-martingale.
\end{prop}
\begin{proof}
Let $X_t$ be a semimartingale in $P$ such that $X= X^h \cdot V$ as was explained above. Take a vertical form $\theta$
on $E$. By Proposition \ref{changeverticalformula},
\[
\int \theta d^{v} X = \int \theta \delta^{v}(X^{h}\cdot V) - \frac{1}{2}\int \nabla^{v*}\theta(dX,\mathbf{v}dX).
\]
A simple account yields
\[
\int \theta \delta^{v}(X^{h}\cdot V) = \int R_{V}^{*}\theta \delta^{v}X^{h} + \int X^{h*}\theta\delta V = \int X^{h*}\theta \delta V,
\]
where we used in last equality the fact that $X^{h}$ be a horizontal process and that every vertical form is the write
as $f\omega$, where $f \in C^{\infty}(P)$ and $\omega$ is the connection form. From this and Lemma above we conclude
that
\[
\int \theta d^{v} X = \int X^{h*} \theta\delta V - \frac{1}{2}\int \nabla^{G*}(X^{h*}\theta)(dV,dV) = \int X^{h*}\theta d^{G}V,
\]
where the last equality is due to change formula between Stratonovich and It\^o integral. Since $\theta$ is arbitrary,
it follows that $X$ is a vertical martingale if and only if $V$ is a $\nabla^{G}$-martingale.\qed
\end{proof}

\end{document}